\documentclass[preprint,12pt]{elsarticle}

\usepackage{amssymb}
\usepackage{amsmath}
\usepackage[
    colorlinks=true,      % Enable colored links instead of boxes
    linkcolor=RoyalBlue4, % Color for internal links (sections, pages, etc.)
    citecolor=RoyalBlue4, % Color for citation links
    urlcolor=RoyalBlue4,  % Color for external URLs
    filecolor=RoyalBlue4  % Color for file links
]{hyperref}

\newcommand{\rr}{\mathbb{R}}

\newcommand{\p}{{\mathbf{P}}}

\def\m{\mathbf{m}}

\def\null{{\rm null}}

\def\dim{{\rm dim}}

\def\pt{\mathbf{p}}

\def\p{\mathbf{P}}
\usepackage{amsthm}

\newtheorem{theorem}{Theorem}

\newtheorem{corollary}{Corollary}

\newtheorem{example}{Example}

\newtheorem{lemma}{Lemma}

\journal{}

\makeatletter
\def\ps@pprintTitle{%
  \let\@oddhead\@empty
  \let\@evenhead\@empty
  \def\@oddfoot{}%
  \def\@evenfoot{}%
}
\makeatother

\begin{document}

\begin{frontmatter}



\title{Trees with extremal Laplacian eigenvalue multiplicity}


\author[label1]{Vinayak Gupta}
\ead{vinayakgupta1729v@gmail.com}
\author[label1]{Gargi Lather}
\ead{gargilather@gmail.com}
\author[label1]{R. Balaji}
\ead{balaji5@iitm.ac.in}

\affiliation[label1]{organization={Department of Mathematics, Indian Institute of Technology Madras},
            city={Chennai},
            postcode={600036}, 
            country={India}}

\begin{abstract}
Let $T$ be a tree.
Suppose $\lambda$ is an eigenvalue of
the Laplacian matrix of $T$ with multiplicity $m_{T}(\lambda)$.
It is known that $m_{T}(\lambda) \leq p(T)-1$,
where $p(T)$ is the number of pendant vertices of $T$.
In this paper, we characterize all trees $T$ for which
there exists an eigenvalue $\lambda$ such that $m_{T}(\lambda)=p(T)-1$.
We show that such trees are precisely either paths, or
there exists an integer $q$ such that if $\alpha$ and $\beta$ are two distinct pendant vertices, then the distance $d(\alpha,\beta)$ satisfies
$d(\alpha, \beta) \equiv 2q ~{\rm{mod}}~(2q+1)$.
As a consequence, we show that $1$ is an eigenvalue of $L_T$ with multiplicity $p(T)-1$ if and only if
$d(\alpha,\beta) \equiv 2\,\mbox{mod}\, 3$ for all distinct pendant vertices $\alpha$ and $\beta$ of $T$.
\end{abstract}

\end{frontmatter}

\section{Introduction}
Let $T$ be a tree. 
Then the eigenvalues of the adjacency and the Laplacian matrices of $T$ can have multiplicity at most $p(T)-1$, where $p(T)$ is the number of pendant vertices of $T$: see \cite[Corollary 2.10]{wang} and \cite[Theorem 2.3]{grone}. 
The trees $T$
for which the adjacency matrices admit an eigenvalue of multiplicity $p(T)-1$ are characterized in
\cite{wong2} and \cite{wong1}. 
In particular, we note that the adjacency matrix
has an eigenvalue with multiplicity $p(T)-1$ if and only if for any two
pendant vertices $u$ and $v$,
there is an integer
$m$ such that $d(u,v)+1 \equiv m\,{\rm mod}\,(m+1)$. 
Inspired by this, we investigate
trees for which the Laplacian matrices have an eigenvalue with multiplicity $p(T)-1$. Let $L_T$ denote the Laplacian matrix of $T$.
If $T$ is a path, then all the eigenvalues of $L_T$
are simple. If $T$ is a star, then $1$ is an eigenvalue of $L_T$ with
multiplicity $p(T)-1$. 
There exist several trees, other than paths and stars, for which 
there exists an eigenvalue of $L_T$
with multiplicity $p(T)-1$. 
Our investigation reveals that, in all such trees
the distances between
any two pendant vertices satisfy a specific combinatorial relation.
The following theorem
summarizes our finding.
\begin{theorem}\label{main1}
Let $T$ be a tree with $p$ pendant vertices. Then the following are equivalent.
 \begin{enumerate}
     \item[\rm{(i)}]  There exists an eigenvalue of $L_T$ with multiplicity $p-1$.
     \item[\rm{(ii)}] Either $T$ is a path or there exists an integer $q$ such that if $\alpha$ and $\beta$ are distinct pendant vertices, then \(d(\alpha,\beta)\equiv 2q\,{\rm{mod}}\,(2q+1).\)
 \end{enumerate}
\end{theorem}
In the sequel, we
show that if $\lambda$ is an eigenvalue of $L_T$ with multiplicity $p(T)-1$,
then $\lambda$ can be written as $2(1- \cos \theta)$ for some $\theta$, and hence $\lambda \leq 4$. In view of Theorem \ref{main1}, all trees $T$ of order up to eight for which $L_T$ has an eigenvalue with multiplicity $p(T)-1$ are listed below. Let $P_n$ and $S_n$ denote the path and the star with $n$ vertices.

\begin{figure}[ht!]
\centering
\includegraphics[scale=1]{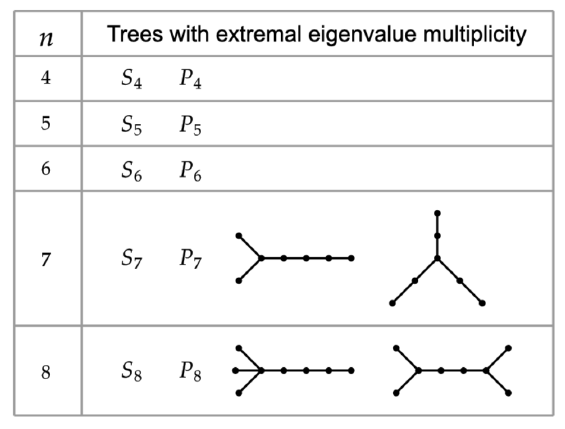}
\label{tabdiag}
\end{figure}

Some significant results are known in the literature
regarding the multiplicity of integral Laplacian eigenvalues.
If $T$ is a tree with $n$ vertices, then
Grone et al. \cite{grone} have shown that any integral eigenvalue of $L_T$ more than {\it one} is simple and divides $n$. 
A well-known result of Faria \cite{faria} states that
if $G$ is a connected graph with $1$ as a Laplacian eigenvalue, then
its multiplicity is at least $p-q$,
where $p$ and $q$ are the number of pendant and quasi-pendant vertices of $G$. 
Barik et al. \cite{barik} have shown that there exist
infinitely many trees distinct from paths for which $1$ is a Laplacian eigenvalue. 
Using Theorem \ref{main1}, we
characterize all trees $T$ for which
$1$ is an eigenvalue of $L_T$ with multiplicity $p(T)-1$. 
We show that such trees are precisely those that satisfy the relation
$d(\alpha,\beta) \equiv 2~\mbox{mod}\,3$ for all distinct pendant vertices
$\alpha$ and $\beta$.

\subsection{Preliminaries}
Let $T$ be a tree. We use the following notation.
\begin{enumerate}
    \item[\rm (i)]  We denote by $V(T)$ and $E(T)$, the set of all vertices and edges of $T$.
Let $u$ and $v$ be distinct vertices in $V(T)$. If $u$ and $v$ are adjacent in $T$, then we write
$u \sim v$. Let $d(u,v)$ be the distance between $u$ and $v$. The path connecting $u$ and $v$
will be denoted by $\p_{uv}$. If the edges of the path $\p_{uv}$ are 
$(u,a_2),(a_2,a_3),\dotsc,(a_k,v)$, then we write $u \sim a_2 \sim \cdots \sim a_k \sim v$.
\item[\rm (ii)] 
If the degree of a vertex in $T$ is at least $3$,
then it is called a major vertex of $T$. The set of all major vertices
of $T$ will be denoted by $\m(T)$. The set of all the pendant vertices
of $T$ will be denoted by $\pt(T)$.
\item[\rm (iii)] 
Let $L_T$ be the Laplacian matrix of $T$. The multiplicity of an eigenvalue 
$\lambda$ of $L_T$ is denoted by
$m_{T}(\lambda)$. 
\item[\rm (iv)] All vectors in the Euclidean space $\rr^k$ will be considered as
column vectors. If $x$ is a vector in $\rr^k$, then we denote its
$j$th component by $x_j$.
\item[\rm (v)] Let $T_1$ and $T_2$ be any two trees such that $|V(T_1)\cap V(T_2)|=1$.
Then $T_1 \circ T_2$ will be the tree with vertices $V(T_1) \cup V(T_2)$
and edges $E(T_1) \cup E(T_2)$. 
\end{enumerate}

Let $\p$ be a path. Then the
eigenvalues and eigenvectors of $L_{\p}$ can be explicitly computed.
We shall make frequent use of the following result from \cite{notes} in our proofs.
\begin{theorem}\label{th1}
    Let $\p$ be the path $1 \sim 2 \sim \dotsc \sim n-1 \sim n$.
    Then,
    \(L_\p x^j=\lambda_j x^j,\)
    where 
    \begin{equation} \label{pn}
    \lambda_j:=2\left(1-\cos \frac{\pi j}{n}\right) ~~~j=0:n-1;
    \end{equation}
    and the vector $x^j=(x^j_{1},\dotsc,x^j_{n})'$ is given by
    \begin{equation} \label{xj}
    x^j_\nu:=\cos \left(\frac{\pi j}{n} \left (\nu-\frac{1}{2}\right ) \right) ~~~\nu=1:n.
    \end{equation}
\end{theorem}

\section{Lemmas}
The proof of Theorem \ref{main1} will be accomplished by
a sequence of lemmas.
As a consequence of Theorem \ref{th1}, we first note the following.
\begin{lemma}\label{lem1}
Let $\p$ and $x_{\nu}^j$ be as in Theorem \ref{th1} and
$l \in \{0,\dotsc,n-1\}$ be such that
$\frac{l}{n}=\frac{2b+1}{2q+1}$,
where $b$ and $q$ are some integers.
If $u$ and $v$ are any two vertices of $\p$ such that
    $d(v,u) \equiv 0\,{\rm{mod}}\,(2q+1)$,
then $x^l_v=0$ if and only if $x^l_u=0$. 
\end{lemma}

\begin{proof}
Let $u$ and $v$ be vertices of $\p$ such that $d(u,v) \equiv 0\,{\rm mod}\, (2q+1)$.
Then, there exists an integer $N$ such that $d(u,v)=(2q+1)N$. Hence,
\begin{equation} \label{uv}
u=v \pm (2q+1)N.
\end{equation}
Define $s:=\frac{2b+1}{2q+1}\left(v-\frac{1}{2}\right)$.
Then, $x^{l}_v=\cos (s \pi)$ and
    \begin{align}\label{xju}
x^l_u&=\cos \left(\frac{\pi l}{n}\left(u-\frac{1}{2}\right)\right) \notag \\
&=\cos \left(\frac{\pi l}{n} \left(v-\frac{1}{2}\right) \pm \frac{\pi l}{n}\left(2q+1\right)N\right)~~~(\mbox{by}~\eqref{uv})  \\
&=\cos(s \pi \pm (2b+1)N \pi). \notag 
\end{align}

Let $m:=(2b+1)N$.
    By \eqref{xju},
 $x^{l}_u= \cos (s \pi \pm m \pi)$.
    Thus, $x^l_u=(-1)^m \cos (s \pi)$.
        Therefore, $x^l_u=0$ if and only if $x^l_v=0$.
The proof is complete.
\end{proof}

\begin{lemma} \label{b1}
Let $T$ be a tree with vertices $1,\dotsc,n$.
Let $u \in \pt(T)$ and $v$ be adjacent to $u$.
Assume that $L_u$ is the Laplacian matrix of $T \smallsetminus (u)$. 
Let $x=(x_1,\dotsc,x_n)'$ be a non-zero vector such that $L_Tx=\lambda x$. If $x_u=0$ and
$y=(x_1,\dotsc,x_{u-1},x_{u+1},\dotsc,x_n)'$,
then $x_v=0$ and $L_uy=\lambda y$.
    \end{lemma}
\begin{proof}
Without loss of generality, assume that $u=n$ and $v=1$.
Then, the vertex $n$ is pendant, $x_n=0$ and $1 \sim n$.
The $n$th row of $L_T$ is $(-1,0,\dotsc,0,1)$.
Hence, the $n$th equation in $L_{T}x=\lambda x$ is 
$x_n-x_1=\lambda x_n$.
Since $x_n=0$, we see that $x_1=0$. 
Define $y:=(x_1,\dotsc,x_{n-1})'$. As
$x_1=0$, $x_n=0$ and $L_Tx=\lambda x$, we get
$L_u y=\lambda y$. 
The proof is complete.
\end{proof}

\begin{lemma}\label{b2}
Let $T$ be a tree with vertices $1,\dotsc,n$ and
$\lambda$ be an eigenvalue of $L_T$.
If $\Omega \subseteq \{1,\dotsc,n\}$ and $|\Omega|=m_T(\lambda)-1$, then there exists a non-zero vector $x=(x_1,\dotsc,x_n)'$ such that 
    \[L_Tx=\lambda x ~~\mbox{and}~~
x_j=0 ~~~ j \in \Omega.\]
\end{lemma}

\begin{proof}
Let \[W:=\{(y_1,\dotsc,y_n)' \in \rr^n: y_j=0~~\forall j \in \Omega \}.\]
Since ${\dim}(W)+m_T(\lambda)=n+1$, the subspace
$\widetilde{W}:=W \cap \null(\lambda I -L_T)$ is non-trivial.
If $0 \neq (x_1,\dotsc,x_n)' \in \widetilde{W}$, then
$x_j=0$ for all $j \in \Omega$ and $L_Tx=\lambda x$. The proof is complete. 
\end{proof}

\begin{lemma} \label{newlem5}
Let $T$ be a tree with vertices $1,\dotsc,n$.
Assume that $T$ has at least one major vertex
and $\pt(T)=\{1,\dotsc,p\}$. Let the path $\p$ connecting $1$ and $2$ in $T$ be
given by \[1 \sim p+1 \sim p+2 \sim \cdots \sim p+k \sim 2.\]
If there exists a non-zero vector $x:=(x_1,\dotsc,x_n)'$ such that
\[L_Tx=\lambda x~~\mbox{and}~~x_{j}=0~~~\mbox{for all}~~ j \in \pt(T) \smallsetminus \{1,2\}\] 
then
\item[\rm (i)] $x_v=0$ $ \mbox{for all}~~ v \in \m(T)$.
\item[\rm (ii)] $x_j=0$ $ \mbox{for all}~~ j \notin V(\p)$.
\item[\rm (iii)] Let $f:=(x_1,x_2,x_{p+1},\dotsc,x_{p+k})'$. Then, $f$ is non-zero and
$L_{\p}f=\lambda f$.
\end{lemma}
\begin{proof}
We prove by induction on $n$. 
If $n=4$, then the result is immediate.
Let $n\geq 5$. Assume the result for all trees with fewer than $n$ vertices. 
Let $T$ have $n$ vertices.  Assume that $T$ has a major vertex and
$\pt(T)=\{1,\dotsc,p\}$.
Suppose $x:=(x_1,\dotsc,x_n)'$ is a non-zero vector such that
$L_Tx=\lambda x$ and $x_j=0$ for all $j \in \pt(T) \smallsetminus \{1,2\}$.
Let $\theta \sim 3$. 
Since $x_3=0$,
by Lemma \ref{b1}, we get
$x_{\theta}=0$.
Define
$T_3:=T \smallsetminus (3)$.
Then, $\m(T) \subseteq \m(T_3) \cup \{\theta\}$.
By Lemma $\ref{b1}$,
$L_{T_3}(x_1,x_2,x_4,\dotsc,x_n)'=\lambda (x_1,x_2,x_4,\dotsc,x_n)'$.
Since $\pt(T_3) \subseteq \{\theta\}\cup\{1,\dotsc,p\} \smallsetminus \{3\}$,
$x_{j}=0$ for all $j \in \pt(T_3) \smallsetminus \{1,2\}$. 
Thus, $T_3$ satisfies the induction hypothesis. 
Hence, 
$x_j=0$ for all $j \in \m(T_3)$.
Since
$\m(T) \subseteq \m(T_3) \cup \{\theta\}$, it follows that
$x_j=0$ for all $j \in \m(T)$.
This proves (i).
The proofs of (ii) and (iii) follow immediately from the induction hypothesis.
This completes the proof.
\end{proof}

\begin{lemma}\label{v1}
    Let $T_1$ and $T_2$ be trees with vertices $1,\dotsc,n$ and 
    $n,\dotsc,n+m$, respectively. If
  $y=(y_1,\dotsc,y_{n-1},0)'$ is a non-zero vector in $\rr^{n}$ such that
  $L_{T_1}y=\lambda y$, then $L_{T_{1} \circ T_2}(y_1,\dotsc,y_{n-1},0,\dotsc,0)'=\lambda (y_1,\dotsc,y_{n-1},0,\dotsc,0)'$.
\end{lemma}

\begin{proof}
 Let $L_{T_1}:=(\alpha_{ij})$ and $L_{T_1 \circ T_2}:=(\beta_{ij})$.
 Define 
\[\widetilde{y}:=(y_1,\dotsc,y_{n-1},y_n,\dotsc,y_k)',\]
where $k:=m+n$ and $y_j=0$ for $j\geq n$.
We need to show that 
\[\sum_{j=1}^{k} \beta_{ij}y_j=\lambda y_i~~\mbox{for all}~~ i=1:k.\]
Since $L_{T_1}y=\lambda y$ and $y_n=0$, we get
\begin{equation}\label{lem41}
   \sum_{j=1}^{n-1}\alpha_{ij}y_j=\lambda y_i ~~~i=1:n. 
\end{equation}
Let $(w_1,\dotsc,w_k)':=L_{T_1 \circ T_2} \widetilde{y}$. 
Then, \[w_i=\sum_{j=1}^{k}\beta_{ij}y_j~~~i=1:k.\] 
Since $y_j=0$ for all $j=n:k$, 
\begin{equation}\label{lem42}
w_i=\sum_{j=1}^{n-1}\beta_{ij}y_j~~~i=1:k.
\end{equation}
If $i \in \{1,\dotsc,n\}$ and $j \in \{1,\dotsc,n-1\}$, then
$\alpha_{ij}=\beta_{ij}$. If $i \in \{n+1,\dotsc,k\}$ and
$j \in \{1,\dotsc,n-1\}$, then $\beta_{ij}=0$.
Thus, by \eqref{lem41} and \eqref{lem42},
\[w_i=
\begin{cases}
    y_i ~~&i=1:n \\
    0  ~~&i=n+1:k.
\end{cases}
\]
Thus, $L_{T_{1} \circ T_2} \widetilde{y}=\lambda \widetilde{y}$. This concludes the proof.   
\end{proof}

\begin{lemma} \label{path}
Let $G$ be the path
\[1:=u \sim 2 \sim \cdots \sim k_1+1:=v \sim k_1+2 \sim \cdots \sim k_1+k_2 \sim k_1+k_2+1:=w.\]
If there exists a positive integer $q$ such that 
$d(u,v)\equiv q\,{\rm mod}\,(2q+1)$ 
and $d(w,v)\equiv q\,{\rm mod}\,(2q+1)$, then
for any integer $b \in [0,q)$,
\[\lambda:=2\left(1-\cos \left(\frac{2b+1}{2q+1}\pi\right)\right)\]
is an eigenvalue of $L_G$. Furthermore, there exists an eigenvector
$y$ corresponding to $\lambda$ such that $y_v=0$ and $y_u=\cos\left( \left(\frac{2b+1}{2q+1}\right)\frac{\pi}{2} \right).$
\end{lemma}
\begin{proof}
We note that
$d(u,v)=k_1$ and $d(w,v)=k_2$. Since $k_j\equiv q\,{\rm{mod}}\,(2q+1)$, there exist integers $N_1$ and $N_2$ such that
\begin{equation} \label{kj}
k_j=(2q+1)N_j+q~~~~j=1,2.
\end{equation}
Define
$k:=k_1+k_2+1=(2q+1)(N_1+N_2+1)$.
Let $b$ be an integer in $[0,q)$ and
$\gamma:=\frac{1}{k} (2b+1)(N_1+N_2+1)$. Then,
 \(\gamma=\frac{2b+1}{2q+1}.\)
Define
\[\lambda:=2\left(1-\cos (\gamma \pi) \right)~~\mbox{and}~~y_j:=\cos\left(\gamma \left(j-\frac{1}{2} \right)\pi \right)~~~j=1:k.\]
By Theorem \ref{th1},
\(L_G(y_1,\dotsc,y_{k})'=\lambda (y_1,\dotsc,y_{k})',\)
where
\[y_1=\cos\left(\frac{\gamma \pi}{2}\right)~~\mbox{and}~~y_v=
        \cos \left(\gamma\left( k_1+\frac{1}{2} \right)\pi  \right).\]
By \eqref{kj}, $$y_v= \cos \left( \frac{(2b+1)(2N_1+1)}{2}\pi \right).$$
        Thus,  $y_v=0$. This completes the proof.
        \end{proof}

\begin{lemma}\label{equiv}
  Let $T$ be a tree with at least one major vertex and $q$ be an integer. Then, the following are equivalent.
  \begin{enumerate}
      \item[\rm{(i)}] $d(u,w) \equiv 2q\,{\rm{mod}}\,(2q+1)$ for all distinct $u,w \in \pt(T)$.
      \item[\rm{(ii)}] $d(u,\theta) \equiv q\, {\rm{mod}}\, (2q+1)$ for all $u \in \pt(T)$ and $\theta \in \m(T)$.
  \end{enumerate}
  \end{lemma}

\begin{proof}
 We prove (i)$\Rightarrow$(ii). Let $u\in \pt(T)$ and $\theta\in \m(T)$. We claim that $d(u,\theta)\equiv q\,{\rm{mod}}\,(2q+1)$. Since $\theta$
 is a major vertex, we can find two pendant vertices $u_2$ and $u_3$ such that $\theta$ lies on the paths $\p_{uu_2}$, $\p_{uu_3}$ and $\p_{u_2u_3}$. 
 Set
 \[k_1:=d(u,\theta),~~k_2:=d(u_2, \theta)~~~\mbox{and}~~k_3:=d(u_3,\theta).\] 
 By (i),
 \begin{equation}\label{(1)}
     k_i+k_j \equiv 2q\,{\rm{mod}}\,(2q+1)~~~i,j=1:3.
     \end{equation}
Thus, \[2(k_1+k_2+k_3) \equiv 6q\, {\rm{mod}}\, (2q+1).\]  Since ${\rm{gcd}}(2,2q+1)=1$,
\begin{equation}\label{(2)}
    k_1+k_2+k_3\equiv 3q\,{\rm{mod}}\, (2q+1).
\end{equation}
By \eqref{(1)} and \eqref{(2)}, 
\[k_i\equiv q\,{\rm{mod}}\,(2q+1)~~~i=1:3.\]
This proves (ii).

 We now prove (ii) $\Rightarrow$ (i). Let $u$ and $w$ be two distinct vertices in $\pt(T)$ and 
$\theta \in \m(T)$ be such that $\theta \in V(\p_{uw})$. Then,
\(d(u,w)=d(u,\theta)+d(\theta,w).\)
By our assumption,
\(d(u,w)\equiv 2q\,{\rm{mod}}\,(2q+1).\)
The proof is complete.
\end{proof}

\begin{lemma}\label{lem6base}
Let $T$ be a tree with exactly one major vertex $v$
and $p$ pendant vertices.
Assume that there exists an integer $q$ 
such that
\[d(u_i,u_j)\equiv 2q\,{\rm{mod}}\,(2q+1)~~\mbox{for all distinct}~~u_i,u_j \in \pt(T).
\]
If $b$ is any integer such that $b \in [0,q)$, then
\[\lambda:=2\left(1- \cos \left(\frac{2b+1}{2q+1}\pi\right) \right)\]
is an eigenvalue of $L_T$. Furthermore,
there exist linearly independent vectors $x^1,\dots,x^{p-1}$
such that 
\begin{enumerate}
\item[\rm (a)] $L_Tx^j=\lambda x^j$ for all $j=1:p-1$.
\item[{\rm (b)}] %If $v$ is the major vertex, then
$x^j_{v}=0$ for all $j=1:p-1$. If $s \in V(T)$ and
$d(s,v) \equiv 0\,{\rm{mod}}\,(2q+1)$, then $x^j_{s}=0$ for all $j=1:p-1$. 
\end{enumerate}
\end{lemma}

\begin{proof}
 Let $n:=V(T)$. We prove the result by induction on $p$. Let $p=3$
and $\pt(T)=\{u_1,u_2,u_3\}$.
Define
$$k_j:=d(u_j,v)~~~j=1:3.$$
By Lemma \ref{equiv},
\[k_j\equiv q\,{\rm{mod}}\,(2q+1)~~~j=1:3.\]
Let $N_1$, $N_2$ and $N_3$ be such that
\[k_j=(2q+1)N_j+q~~~j=1:3.\]
Hence,
\begin{align} \label{kl+kt}
k_1+k_2+1=(2q+1)(N_1+N_2+1).
\end{align}
Let $b$ be an integer such that $b \in [0,q)$. Define
\begin{equation} \label{deltadefn}
\delta:=(2b+1)(N_1+N_2+1)~~\text{and}~~\gamma:=\frac{\delta}{k_1+k_2+1}. 
\end{equation}
Since $0 \leq b < q$,
\begin{equation} \label{deltae}
\delta < (2q+1) (N_1+N_2+1).
\end{equation}
By \eqref{kl+kt} and $\eqref{deltae}$,
$\delta$ is an integer in  the interval $(0,k_1 + k_2 +1)$.
Let the path $\p_{u_1u_2}$ in $T$ be
\[1:=u_1 \sim 2 \sim 3 \sim \cdots \sim k_1+1:=v \sim \cdots \sim k_1+k_2 \sim k_{1}+k_{2}+1:=u_2.\]
In view of Theorem \ref{th1}, 
\[\lambda:=2\left(1- \cos\left(\frac{{\delta \pi}}{{k_1+k_2+1}}\right)\right)\]
is an eigenvalue of $L_{\p_{u_1u_2}}$ and
there exists a vector $y=(y_1^{\delta},\dotsc,y_{k_1+k_2+1}^{\delta})'$ such that
$L_{\p_{u_1u_2}}y=\lambda y$, where 
\begin{equation}\label{eigvec}
 y_j^{\delta}:=\cos\left(\frac{\delta}{k_1+k_2+1} \left(j-\frac{1}{2} \right)\pi \right)~~~j=1:k_1+k_2+1.   
\end{equation}
 Since $k_1+1=v$, by \eqref{deltadefn} and \eqref{eigvec}, we obtain 
\[y_v^{\delta}=y_{k_1+1}^{\delta}=
        \cos \left(\gamma\left( k_1+\frac{1}{2} \right)\pi  \right).\]
As $k_1=(2q+1)N_1+q$ and $\gamma=\frac{2b+1}{2q+1}$,
       \[y_v^{\delta}= \cos \left( \frac{(2b+1)(2N_1+1)}{2}\pi \right).\]
        Thus,  $y_v^{\delta}=0$. Let $\p_{vu_3}$ be the path from $v$ to $u_3$.
        Then,
        $T=\p_{u_1u_2} \circ \p_{v u_3}$.
Define $f=(f_1,\dotsc,f_n)'$ by
\begin{equation}\label{xr}
f_r:=
\begin{cases}
y_r^{\delta} & r\in V(\p_{u_1u_2}) \\
0 & \mbox{else.}
\end{cases}    
\end{equation}
Since $y_v^\delta=0$, Lemma \ref{v1} implies that
$L_Tf=\lambda f$.
As $f_v=y_v^{\delta}$ and $y_v^{\delta}=0$, we get $f_v=0$.
Let $s \in V(T)$ be  such that $d(s,v) \equiv 0\,{\rm{mod}}\,(2q+1)$. 
We have two cases.
If $s\notin V(\p_{u_1u_2})$, then
by \eqref{xr}, $f_s=0$. If $s\in V(\p_{u_1u_2})$, then by Lemma \ref{lem1}, 
$y_s^{\delta}=0$, and again by
\eqref{xr},
\(f_s=0.\)
Thus, $f$ satisfies (a) and (b).

Applying the same argument to the path $\p_{u_2u_3}$,
we obtain a non-zero vector $g=(g_1,\dotsc,g_n)'$ such that
\begin{enumerate}
\item[(i)] $L_Tg=\lambda g$ 
\item[(ii)] $g_v=0$ and $g_i=0$ for all $i\notin V(\p_{u_2u_3})$ 
\item[(iii)] $g_s=0$ for all $s$ such that
$d(s,v)\equiv 0\,{\rm{mod}}\,(2q+1)$.
\end{enumerate}
Thus, $g$ satisfies (a) and (b).

We now claim that $f$ and $g$ are linearly independent.
Since
\[f_{u_1}=f_1=\cos\left(\frac{\gamma \pi}{2} \right) ~~\mbox{and} ~~0<\gamma<1,\]
we deduce that $f_1 \neq 0$.
Because
 $1=u_1 \notin V(\p_{u_2u_3})$, we see that
$g_1=0$.
Therefore, $f$ and $g$ are linearly independent.
Hence, the result is true for $p=3$.
We shall assume the result for trees with one major vertex and at most $p-1$ pendant vertices. 

Let $p>3$ and
$T$ be a tree with
$p$ pendant vertices. Let $u_1$ and $u_2$ be two distinct vertices in $\pt(T)$. Let the path $\p_{u_1u_2}$ be 
\[1:=u_1 \sim 2 \sim  \cdots \sim k_1+1:=v \sim \cdots \sim k_1+k_2+1:=u_2.\]

Let $H$ be the component of $T\smallsetminus ({k_1})$ such that $v\in V(H)$. 
Then, $|\pt(H)|=p-1$. 
Since $p>3$, $H$ has a major vertex. 
By the induction hypothesis,
there exist linearly independent vectors $y^1,\dots,y^{p-2}$ such that
\begin{equation} \label{lh}
L_Hy^j=\lambda y^j,~~y^{j}_v=0~~\mbox{and}~~y^{j}_s=0~~~j=1:p-2
\end{equation}
for all $s \in V(H)$ such that
$d(s,v)\equiv 0\,{\rm{mod}}\,(2q+1)$.
Define $x^1,\dots,x^{p-2}$ in $\mathbb{R}^n$ as follows:
\begin{equation}\label{xi}
x^j_r:=
\begin{cases}
y^j_r & r\in V(H) \\
0 & \mbox{else.}
\end{cases}    
\end{equation}
As \[T=H\circ \p_{u_1v}~~\mbox{and}~~
y^j_v=0 ~~~ j=1:p-2;\] 
by Lemma \ref{v1}, 
\[L_T x^j=\lambda x^j~~~j=1:p-2.\]
In view of $\eqref{lh}$ and \eqref{xi},
\[x^j_v=0 ~~\mbox{and}~~x^j_s=0~~~j=1:p-2\]
for all $s\in V(T)$
such that $d(s,v) \equiv 0\,{\rm{mod}}\,(2q+1)$.
As $y^1,\dotsc,y^{p-2}$ are 
linearly independent, $x^1,\dotsc,x^{p-2}$
are linearly independent.

We now construct another eigenvector $x^{p-1}$. 
By Lemma \ref{equiv},
$$k_j\equiv q\,{\rm{mod}}\,(2q+1)~~j=1,2.$$ 
Hence, there exist integers $N_1$ and $N_2$ such that 
\[k_1=(2q+1)N_1+q~~\mbox{and}~~k_2=(2q+1)N_2+q.\] 
Therefore,
\begin{equation}\label{ksum}
    k_1+k_2+1=(2q+1)(N_1+N_2+1).
\end{equation}
Let $b$ be an integer in $[0,q)$. Put
\begin{equation} \label{deltadefn1}
\delta:=(2b+1)(N_1+N_2+1)~~\text{and}~~\gamma:=\frac{2b+1}{2q+1}. 
\end{equation}
By Theorem \ref{th1}, 
\[L_{\p_{u_1u_2}}(h_1^{\delta},\dotsc,h_{k_1+k_2+1}^{\delta})'=\lambda (h_1^{\delta},\dotsc,h_{k_1+k_2+1}^{\delta})',\]
where
\begin{equation}\label{hjdel}
 \lambda:=2\left(1-\cos (\gamma \pi) \right)~~\mbox{and}~~h_j^{\delta}:=\cos\left(\frac{\delta}{k_1+k_2+1} \left(j-\frac{1}{2} \right)\pi \right).   
\end{equation}
In particular, 
\[h_v^{\delta}=h_{k_1+1}^{\delta}=
        \cos \left(\gamma\left( k_1+\frac{1}{2} \right)\pi  \right).\]
Since $k_1=(2q+1)N_1+q$, we have
       $$h_v^{\delta}= \cos \left( \frac{(2b+1)(2N_1+1)}{2}\pi \right).$$
        Thus,  
        \begin{equation} \label{hv0}
            h_v^{\delta}=0.
        \end{equation}
        Define $x^{p-1}\in \rr^n$ by
\begin{equation}\label{p-1}
x^{p-1}_r:=
\begin{cases}
h_r^{\delta} & r\in V(\p_{u_1u_2}) \\
0 & \mbox{else.}
\end{cases}    
\end{equation}
Since $h_v^{\delta}=0$, using Lemma \ref{v1}, we obtain
$L_Tx^{p-1}=\lambda x^{p-1}$.
By \eqref{hv0} and \eqref{p-1},
$x^{p-1}_v=0$. Let $s\in V(T)$ be such that $d(s,v)\equiv 0\,{\rm{mod}}\,(2q+1)$. 
By Lemma  \ref{lem1}, 
$h_s^{\delta}=0$ for all $s\in V(\p_{u_1u_2})$.
Since $h_s^{\delta}=x_s^{p-1}$, we get $x_s^{p-1}=0$. If $s\notin V(\p_{u_1u_2})$, then \eqref{p-1} implies that $x_s^{p-1}=0$. Thus, $x^{p-1}$ satisfies condition (b).
By \eqref{xi},
\begin{equation}\label{xu1pi}
 x^j_{u_1}=0~~\mbox{for all}~~j=1:p-2.   
\end{equation}
By \eqref{hjdel} and \eqref{p-1},
\begin{equation*}
  x^{p-1}_{u_1}= \cos\left(\left( \frac{2b+1}{2q+1}\right)\frac{\pi}{2} \right). 
\end{equation*}
Since $0<\frac{2b+1}{2q+1}<1$, 
\begin{equation}\label{xjup-1}
  x^{p-1}_{u_1}\neq 0. 
\end{equation}
By \eqref{xu1pi} and \eqref{xjup-1}, $x^1,\dotsc,x^{p-1}$ are linearly independent.
The proof is complete.
\end{proof}

\begin{lemma}\label{sufficient}
Let $T$ be a tree with at least one major vertex and $p$ pendant vertices.
If there exists an integer $q$ such that
$d(\alpha,\beta)\equiv 2q\,{\rm{mod}}\,(2q+1)$ for all distinct $\alpha,\beta \in \pt(T)$, then
for any integer $b\in[0,q)$, 
$$\lambda:=2\left(1-\cos \left(\frac{2b+1}{2q+1}\pi\right) \right)$$
is an eigenvalue of $L_T$. Associated with each $\lambda$, there exist 
linearly 
independent eigenvectors $x^1,\dots,x^{p-1}$ such that
\begin{enumerate}
    \item[{\rm{(a)}}] If $m \in \m(T)$, then $x^{j}_m=0$ for all $j=1:p-1$.
     \item[{\rm{(b)}}] If $s \in V(T)$, $m \in \m(T)$ and $d(s,m) \equiv 0\,{\rm{mod}}\,(2q+1)$, then $x^j_s=0$ for all $j=1:p-1$.
\end{enumerate}
\end{lemma}

\begin{proof}
We shall use induction on $p$. To get started, we note that by Lemma \ref{lem6base}, the desired holds for $p=3$ 
because then the tree has exactly one major vertex.
We assume the result for all trees with at most $p-1$ pendant vertices. 
Let $T$ be a tree with $p$ pendant vertices and
$q$ be an integer such that if $\alpha$ and $\beta$ are any two distinct vertices
in $\pt(T)$, then
\begin{equation}\label{c(i)}
 d(\alpha,\beta)\equiv 2q\,{\rm{mod}}\,(2q+1).   
\end{equation}
We now choose $u$ and $w$ in $\pt(T)$ such that only one of the vertices in the path $\p_{uw}$
is a major vertex. 
Let \(\m(T) \cap V(\p_{uw})=\{v\}.\)
We assume that the paths  $\p_{vu}$ and $\p_{vw}$ are given by
\[v \sim u_1 \sim \cdots \cdots \sim u_{k_1-1} \sim u ~~\mbox{and}~~
v \sim w_1 \sim \cdots \sim w_{k_2-1} \sim w.\]
Let $H$ be the component of $T\smallsetminus (u_1)$ that
contains the vertex $v$. Then,
$\pt(T)=\pt(H) \cup \{u\}$. 
Hence,
$|\pt(H)|=p-1$.
By \eqref{c(i)},
$d(\alpha,\beta) \equiv 2q\,{\rm{mod}}\,(2q+1)$ for all distinct vertices $\alpha$ and $\beta$ in $\pt(H)$.
Therefore, $H$ satisfies the induction hypothesis. 
Let $b \in [0,q)$. Define
\[\lambda:=2\left(1-\cos\left(\frac{2b+1}{2q+1} \pi\right) \right).
\]
Then, there exist linearly independent vectors $g^1,\dotsc,g^{p-2}$ such that for all $j$
 \begin{enumerate}
     \item[{\rm (a1)}] $L_Hg^j=\lambda g^j$.
     \item[\rm (a2)] $g^j_{\nu}=0$ for all $\nu \in \m(H)$.
     \item[\rm (a3)] $g^j_{\nu}=0$ for all $\nu \in \{s \in V(H):d(s,m) \equiv 0\,{\rm{mod}}\,(2q+1)~\mbox{for some}~m \in \m(H)\}$.
 \end{enumerate}
 We now show that
\begin{equation} \label{gjv=0}
g^j_v=0~~~j=1:p-2.
\end{equation}

If $v \in \m(H)$, then \eqref{gjv=0} follows from (a2).

Suppose $v\notin \m(H)$. Then there are two cases.
If each vertex in $H$ has degree at most two, then $v$
is the only major vertex in $T$ with degree $3$. 
Now, $\eqref{gjv=0}$ follows from Lemma \ref{lem6base}. 
If there exists a vertex in $H$ with degree at least $3$,
then let $z \in \m(H)$. 
Now, $z \in \m(T)$ and the only major vertex of $T$ on the path $\p_{uw}$ is $v$. Thus, $z$ is not on the path $\p_{vw}$.  
By \eqref{c(i)} and Lemma \ref{equiv},
\begin{equation*}
 d(z,w)\equiv q\,{\rm{mod}}\,(2q+1)~~\text{and}~~d(v,w)\equiv q\,{\rm{mod}}\,(2q+1).   
\end{equation*}
Let $N_1$ and $N_2$ be such that
\begin{equation}\label{l11e1}
    d(z,w)=(2q+1)N_1+q~~\text{and}~~d(v,w)=(2q+1)N_2+q.
\end{equation}
Since $v$ must lie on the path $\p_{zw}$,
\begin{equation}\label{l11e2}
d(z,v)=d(z,w)-d(v,w).    
\end{equation}
By \eqref{l11e1} and \eqref{l11e2},
\begin{equation}\label{distvm}
    d(z,v)\equiv 0\,{\rm{mod}}\,(2q+1).
\end{equation}
Thus, \eqref{gjv=0} follows from (a3).

Let $n:=|V(T)|$.
Define $f^1,\dotsc,f^{p-2}$ in $\rr^n$ as follows:
\begin{equation}\label{fj}
f^j_r:=
\begin{cases}
g^j_r & r\in V(H) \\
0 & \mbox{else.}
\end{cases}    
\end{equation}
Since $g^1,\dots,g^{p-2}$ are linearly independent, $f^1,\dots,f^{p-2}$ are also linearly independent.
We observe that
$T=H \circ \p_{vu}$.
By Lemma \ref{v1}, \eqref{gjv=0} and \eqref{fj}, we have
\[L_T f^j=\lambda f^j~~j=1:p-2.\]

By \eqref{c(i)} and Lemma $\eqref{equiv}$,
$$d(v,u)\equiv q\,{\rm{mod}}\,(2q+1)~\text{and}~d(v,w)\equiv q\,{\rm{mod}}\,(2q+1).$$ 
Using Lemma \ref{path}, we get
a vector $y$ such that 
\begin{equation}\label{yv0}
L_{\p_{uw}}y=\lambda y,~~ y_v=0~~\mbox{and}~~y_u=\cos\left( \left(\frac{2b+1}{2q+1}\right)\frac{\pi}{2} \right).    
\end{equation}
Equation \eqref{yv0} and Lemma \ref{lem1} imply that
\begin{equation} \label{d(s,v)}
y_j=0~~\mbox{for all}~j \in \{s \in V(\p_{uw}):d(s,v) \equiv 0\,{\rm{mod}}\, (2q+1)\}.
\end{equation}
Define $f^{p-1}\in \rr^n$, where
\begin{equation}\label{xp-1}
f^{p-1}_r:=
\begin{cases}
y_r & r\in V(\p_{uw}) \\
0 & \mbox{else.}
\end{cases}    
\end{equation}
Let $\widetilde{H}$ be the component of $H \smallsetminus (w_1)$ containing the vertex $v$.
Then, \[T=\widetilde{H} \circ \p_{uw}.\]
Since $y_v=0$, by Lemma \ref{v1}, we get
\[L_Tf^{p-1}=\lambda f^{p-1}.\]
By \eqref{fj},
\begin{equation} \label{fp-10}
f^{j}_u=0 ~~~ j=1:p-2.
\end{equation}
In view of \eqref{yv0} and \eqref{xp-1},
\[f^{p-1}_u=\cos\left(\left( \frac{2b+1}{2q+1}\right)\frac{\pi}{2} \right).\] 
Since $0<\frac{2b+1}{2q+1}<1$, 
\begin{equation} \label{fp-1}
f^{p-1}_u \neq 0.
\end{equation}
Equations
\eqref{fp-10} and \eqref{fp-1} imply that 
$f^1,\dotsc,f^{p-1}$ are linearly independent. 
We now show that 
$f^1,\dotsc,f^{p-1}$ satisfy
(a) and (b). 
By (a2), \eqref{gjv=0} and \eqref{fj}, we see that if $m \in \m(T)$ then,
\[f^j_m=g^j_m=0~~~j=1:p-2.\] Thus, $f^1,\dotsc,f^{p-2}$ satisfy
(a). 
We now show that these vectors satisfy (b).
Let $x \in V(T)$ and $m \in \m(T)$ be such that 
$d(x,m) \equiv 0\,{\rm{mod}}\,(2q+1)$.
We claim that
\begin{equation} \label{conclusion}
f^{j}_x=0~~~j=1:p-2.
\end{equation}
There are two cases.
\begin{enumerate}
    \item[(c1)] Suppose $m\neq v$. If $x\in V(H)$, then 
    \eqref{conclusion} follows from (a3). If $x\notin V(H)$, then $\eqref{conclusion}$ follows
    from \eqref{fj}.
    \item[(c2)] Suppose $m=v$. 
    If $\m(T)=\{v\}$, then $\eqref{conclusion}$ follows as a consequence of Lemma \ref{lem6base}.
    Suppose $\m(T)$ has at least two vertices. Let $z \in \m(T)$ be such that $z$ and $v$ are the only major vertices in
    the path $\p_{zv}$. Then, 
    \begin{equation*}
d(z,x)=
\begin{cases}
 d(z,v)-d(x,v), & x \in V(\p_{zv}) \\
d(z,v)+d(x,v), & x \notin V(\p_{zv}), ~d(v,x)<d(z,x) \\
d(x,v)-d(z,v),  & x \notin V(\p_{zv}), ~d(v,x)>d(z,x).
\end{cases}
\end{equation*}
By \eqref{distvm}, \[d(z,v) \equiv 0\,{\rm{mod}}\,(2q+1).\] We have \[d(x,v)\equiv 0\,{\rm{mod}}\,(2q+1).\]
The above two equations give 
\[d(z,x)\equiv 0\,{\rm{mod}}\,(2q+1).\]
Now, $\eqref{conclusion}$ follows from (c1).   
\end{enumerate}
Thus, $f^1,\dotsc,f^{p-2}$ satisfy (b).
We now show that $f^{p-1}$ satisfies (a) and (b). 
Let $m\in \m(T)$. Suppose $m\neq v$. 
Then, $m\in \m(H)$ and by \eqref{xp-1}, we get $f^{p-1}_m=0$. If $m=v$, then by \eqref{yv0}, $y_v=0$. By \eqref{xp-1}, $f^{p-1}_v=0$.
Thus, $f^{p-1}$ satisfies (a). Now, let $z\in \m(T)$ and $x\in V(T)$ be such that 
\begin{equation}\label{stareqn}
d(x,z)\equiv 0\,{\rm{mod}}\,(2q+1).    
\end{equation}
Then, there are two cases:
\begin{enumerate}
    \item[(A)]  Suppose $z=v$. If $x\in V(\widetilde{H})$, then by \eqref{xp-1}, we obtain $f^{p-1}_x=0$. 
    If $x \notin V(\widetilde{H})$, then $x\in V(\p_{uw})\smallsetminus \{v\}$.    Since $y_v=f^{p-1}_v=0$, by Lemma \ref{lem1}, we get $y_x=f^{p-1}_x=0$.
    \item[(B)] Suppose $z\neq v$. Then $z\in \m(H)$. If $x\in V(\widetilde{H})$, then 
    by \eqref{xp-1}, it follows that
    $f^{p-1}_x=0$. If $x\in V(\p_{uw})\smallsetminus \{v\}$, then $d(v,x)=d(z,x)-d(z,v)$.
    In view of \eqref{distvm} and \eqref{stareqn}, we get $d(v,x)\equiv 0\,{\rm{mod}}\,(2q+1)$. 
    By (A), we get $f^{p-1}_x=0$.
\end{enumerate}
Thus, $f^{p-1}$ satisfies (b). This concludes the proof.
\end{proof}

\begin{lemma} \label{necessary}
 Let $T$ be a tree with at least one major vertex. 
 If $\lambda$ is an eigenvalue of $L_T$ with multiplicity 
 $|\pt(T)|-1$,
 then there exists an integer $q$ such that
 $$d(\alpha,\beta) \equiv 2q\,{\rm{mod}}\,(2q+1)~~\text{for all distinct}~~\alpha, \beta \in \pt(T).$$    
\end{lemma}
\begin{proof}
Let the vertices of $T$ be $1,\dotsc,n$ and $p:=|\pt(T)|$. 
Suppose $\lambda$ is an eigenvalue of $L_T$ such that $m_T(\lambda)=p-1$. 

We claim that there exist integers $a_1$ and $a_2$ such that $\mbox{gcd}(a_1,a_2)=1$ and
if $r$ and $s$ are any two distinct pendant vertices of $T$, then \[\frac{\xi_{rs}}{d(r,s)+1}=\frac{a_1}{a_2},\] 
for some integer $\xi_{rs}$ in the interval $(0,d(r,s)]$. 
    To prove this, we consider two distinct pendant vertices $r$ and $s$. 
Let the path $\p_{rs}$ be  \[1:=r \sim 2 \sim \cdots \sim k+1 \sim \cdots \sim d(r,s)+1:=s.\]
    We assume that $k+1 \in \m(T)$. In view of Lemma \ref{b2}, there exists a non-zero vector $g=(g_1,\dotsc,g_n)' \in \rr^n$ such that 
    \[L_Tg=\lambda g~~\mbox{and}~~g_{j}=0 ~~~ j \in \pt(T)\smallsetminus \{r,s\}.\] 
    By Lemma \ref{newlem5} (i),
    \begin{equation} \label{gjm}
    g_j=0~~~j \in \m(T).
    \end{equation}
    Let \[f_j:=g_j~~~ j=1:s~~\mbox{and}~~
    f:=(f_{1}, f_{2},\dotsc,f_s)'.\]
By Lemma \ref{newlem5} (iii), $f$ is non-zero and
$L_{\p_{rs}}f=\lambda f$. 
In view of Theorem \ref{th1}
\begin{equation}\label{cosrs}
  \lambda=2\left(1-\cos\left( \frac{\pi \xi_{rs}}{d(r,s)+1}\right)\right),
\end{equation}
where $\xi_{rs}$ is some integer in $[0,d(r,s)]$. Since $\lambda$ is non-zero, $\xi_{rs} \in (0,d(r,s)]$. 
Therefore,
\begin{equation}\label{fral}
0<\frac{\xi_{rs}}{d(r,s)+1} < 1.    
\end{equation}
Let $a_1$ and $a_2$ be positive integers such that
 \begin{equation} \label{gcdan}
\mbox{gcd}(a_1,a_2)=1~~\mbox{and}~~\frac{\xi_{rs}}{d(r,s)+1}=\frac{a_1}{a_2}.
\end{equation}
Since $\cos$ is a decreasing function in $(0,1)$, equations \eqref{cosrs} and \eqref{fral} imply that
\eqref{gcdan}
holds for any two distinct pendant vertices. 
This proves the claim.

By Theorem \ref{th1},
\begin{equation}\label{gjk+01}
f_{k+1}=\cos \left(\frac{\pi \xi_{rs}}{d(r,s)+1}\left(k+\frac{1}{2}\right)\right).    
\end{equation} 
Since $f_{k+1}=g_{k+1}$, \eqref{gjm} implies
\begin{equation} \label{gjk+1}
 f_{k+1}=0.
\end{equation}
In view of \eqref{gcdan}, \eqref{gjk+01} and \eqref{gjk+1}, we conclude that there exists an integer $t$ such that 
\begin{equation}\label{gcd1}
 \frac{a_1}{a_2}(2k+1)=2t+1. 
\end{equation}
In view of \eqref{gcdan} and \eqref{gcd1}, we see that 
there exist integers $\theta$ and $\eta$ such that
\begin{equation} \label{aiodd}
a_1=2\theta+1~~\mbox{and}~~a_2=2 \eta+1.
\end{equation}
Equation \eqref{gcdan} implies that
$a_2$ divides $d(r,s)+1$. Therefore, 
\[d(r,s)\equiv (a_2-1)\,{\rm{mod}}\,a_2.\]
By \eqref{aiodd},
\begin{equation*}\label{finaleq}
  d(r,s) \equiv 2\eta\,{\rm{mod}}\,(2\eta+1).  
\end{equation*}
This completes the proof.
\end{proof}

\section{Results}
We now present our findings of this paper. The main result is given first. 

\begin{theorem}\label{main}
 Let $T$ be a tree. If $T$ is not a path, then the following are equivalent.
 \begin{enumerate}
     \item[\rm{(i)}]  There exists an eigenvalue $\lambda$ of $L_T$ with multiplicity $|\pt(T)|-1$.
     \item[\rm{(ii)}]There exists an integer $q$ such that $$d(\alpha,\beta)\equiv 2q\,{\rm{mod}}\,(2q+1)~~\text{for all distinct}~~\alpha,\beta\in \pt(T).$$
 \end{enumerate}
 \end{theorem}

\begin{proof}
Since $T$ is not a path, it has a major vertex. Now, the proof of (i) $\Rightarrow$ (ii) follows from Lemma \ref{necessary} and (ii) $\Rightarrow$ (i) follows from Lemma \ref{sufficient}. This ends the proof.
\end{proof}
We now assert that if $m_{T}(\lambda)=|\pt(T)|-1$, then $\lambda$ has
a special form.
\begin{theorem}\label{th1m}
Let $T$ be a tree with at least one major vertex. If $L_T$ has an eigenvalue $\lambda$ with multiplicity $|\pt(T)|-1$, then there exist integers $b$ and $q$ such that \[\lambda=2\left(1-\cos\left(\frac{2b+1}{2q+1}\pi\right)\right),\] where $b\in (0,q]$.
\end{theorem}

\begin{proof}
The proof follows from equations \eqref{cosrs}, \eqref{gcdan} and \eqref{aiodd}.
\end{proof}

As a consequence, we have the following result.
\begin{corollary}
 Let $T$ be a tree. If $L_T$ has an eigenvalue $\lambda$ with multiplicity $|\pt(T)|-1$, then $\lambda \in [0,4)$.   
\end{corollary}

\begin{proof}
If $T$ is a path, then the proof follows from Theorem \ref{th1}. If $T$ has a major vertex, then the proof follows from Theorem \ref{th1m}. 
The proof is complete.
\end{proof}
The Laplacian matrices of several trees have $1$ as an eigenvalue. 
All trees $T$ such that $m_{T}(1)=|\pt(T)|-1$ have an interesting characterization.

\begin{theorem}\label{cor1}
    Let $T$ be a tree and $p:=|\pt(T)|$. Then, $1$ is an eigenvalue of $L_T$ with multiplicity $p-1$ if and only if $d(\alpha,\beta) \equiv 2\,{\rm{mod}}\,3$ for all distinct pendant vertices $\alpha$ and $\beta$.
\end{theorem}
\begin{proof}
If $T$ is a path, then the result follows from Theorem \ref{th1}. Suppose $T$ has at least one major vertex. 
Let $1$ be an eigenvalue
of $L_T$ with multiplicity $p-1$. 
By Theorem \ref{main}, there exists an integer $q$ such that
for any two distinct pendant vertices $\alpha$ and $\beta$,
\begin{equation}\label{th51}
d(\alpha,\beta)\equiv 2q\,{\rm{mod}}\,(2q+1).
\end{equation}
By Theorem \ref{th1m}, there exists an integer $b \in [0,q)$ such that
\begin{equation}\label{th52}
  1=2\left( 1-\cos\left(\frac{2b+1}{2q+1}\pi \right) \right). 
\end{equation}
Since $0\leq b<q$, \eqref{th52} holds if and only if 
\[\frac{2b+1}{2q+1}=\frac{1}{3}.\] Thus, $3$ divides $(2q+1)$. Hence,
\begin{equation}\label{th53}
 (2q+1) \equiv 0\,{\rm{mod}}\,3.   
\end{equation}
By \eqref{th51} and \eqref{th53}, it follows that for all distinct pendant vertices $\alpha$ and $\beta$,  
$$d(\alpha,\beta)\equiv 2\,{\rm{mod}}\,3.$$

The converse is obtained by specializing $q=1$ 
and $\lambda=2\left(1-\cos\frac{\pi}{3}\right)=1$ in
Lemma \ref{sufficient}. This concludes the proof. 
\end{proof}

\begin{figure}[ht!]
\centering
\includegraphics[scale=0.3]{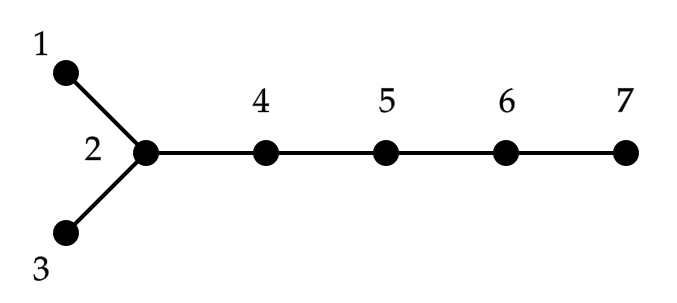}
\caption{$T$}\label{diagp}
\end{figure}

\begin{example}
{\em{Let $T$ be the above tree. Then, $|\pt(T)|=3$. If $\alpha$ and $\beta$ are any two distinct vertices of $T$, then $d(\alpha,\beta)\equiv 2\,{\rm{mod}}\,3$. By Theorem \ref{cor1}, $1$ is an eigenvalue of $L_T$ with multiplicity two.}} 
\end{example}

\section*{Declaration of competing interests}
\noindent
We declare that we have no known competing financial interests or personal relationships that could have appeared to influence the work reported in this paper.

\section*{Funding}
\noindent
This research did not receive any specific grant from funding agencies in the public, commercial, or not-for-profit sectors.

\end{document}